\documentclass[reqno]{amsart}
\usepackage{hyperref}
\usepackage{amssymb}

\allowdisplaybreaks

\theoremstyle{plain}
\newtheorem{theorem}{Theorem}[section]

\newtheorem{lemma}[theorem]{Lemma}

\theoremstyle{definition}
\newtheorem{definition}[theorem]{Definition}

\theoremstyle{remark}
\newtheorem{remark}[theorem]{Remark}
\numberwithin{equation}{section}

\begin{document}

\title[Blow-up problem for nonlocal parabolic equation]
{Blow-up problem for nonlinear nonlocal parabolic equation with absorption under  nonlinear nonlocal boundary condition}

\author[A. Gladkov]{Alexander Gladkov}
\address{Alexander Gladkov \\ Department of Mechanics and Mathematics
\\ Belarusian State University \\  4  Nezavisimosti Avenue \\ 220030
Minsk, Belarus and Peoples' Friendship University of Russia (RUDN University), 6 Miklukho-Maklaya street, Moscow 117198, Russian Federation }    \email{gladkoval@bsu.by}


\begin{abstract}
In this paper we consider initial boundary value problem for
 nonlinear nonlocal parabolic equation with absorption under
nonlinear nonlocal boundary condition and nonnegative initial datum.
We prove comparison principle, global existence and blow-up of solutions.

\end{abstract}

\maketitle

\section{Introduction}

In this paper we consider the initial boundary value problem for
nonlinear nonlocal parabolic equation
\begin{equation}\label{v:u}
    u_t= \Delta u + a u^p \int_{\Omega} u^q (y,t) \, dy - b u^m,\;x\in\Omega,\;t>0,
\end{equation}
with nonlinear nonlocal boundary condition
\begin{equation}\label{v:g}
\frac{\partial u(x,t)}{\partial\nu}=
\int_{\Omega}{k(x,y,t)u^l(y,t)}\,dy, \; x\in\partial\Omega, \; t > 0,
\end{equation}
and initial datum
\begin{equation}\label{v:n}
    u(x,0)=u_{0}(x),\; x\in\Omega,
\end{equation}
where $a,\, b,\, p,\,q, \,m,\,l$ are positive numbers, $\Omega$ is a bounded domain in $\mathbb{R}^N$
for $N\geq1$ with smooth boundary $\partial\Omega$, $\nu$ is unit
outward normal on $\partial\Omega.$

Throughout this paper we suppose that the functions
$k(x,y,t)$ and $u_0(x)$ satisfy the following conditions:
\begin{equation*}
k(x, y, t)\in
C(\partial\Omega\times\overline{\Omega}\times[0,+\infty)),\;k(x,y,t)\geq0;
\end{equation*}
\begin{equation*}
u_0(x)\in C^1(\overline{\Omega}),\;u_0(x)\geq0\textrm{ in
}\Omega,\;\frac{\partial u_0(x)}{\partial\nu}=\int_{\Omega}{k(x,
y,0)u_0^l(y)}\,dy\textrm{ on }\partial\Omega.
\end{equation*}
Various phenomena in the natural sciences and engineering lead to the nonclassical mathematical models subject to nonlocal boundary conditions.
For global existence and blow-up of solutions for parabolic equations and systems
with nonlocal boundary conditions we refer to \cite{CL,D,F,GK,GN1,GN2,KT,K,KD,MV,Pao,WMX} and the references therein.
In particular, the blow-up problem for parabolic equations with nonlocal boundary condition
\begin{equation*}
    u(x,t)=\int_{\Omega}k(x,y,t)u^l(y,t)\,dy,\;x\in\partial\Omega,\;t>0,
\end{equation*}
was considered in~\cite{CYZ,FZ2,GG,GG1,Liu}. Initial boundary value problems for parabolic equations with
nonlocal boundary condition (\ref{v:g}) were studied in~\cite{GK1,GK2,LLLW,LWSL,LHZ}.
So, the problem~(\ref{v:u})--(\ref{v:n}) with $a = 0$ was investigated in~\cite{G1,G2}.
Initial-boundary value problems  for nonlocal parabolic equations
with nonlocal boundary conditions were addressed in many papers also
(see, for example, \cite{CY,FZ3,GK4,LL,YF,YX,ZK}). In particular, some global existence and blow-up results for
(\ref{v:u})--(\ref{v:n}) were obtained in \cite{WY}.

The aim of this paper is to investigate global existence and blow-up of solutions of~(\ref{v:u})--(\ref{v:n}).

This paper is organized as follows. In the next section we prove
comparison principle. The global existence of solutions for any initial data is proved in Section 3. In Section 4 we
present finite time blow-up results for solutions with large initial data.

\section{Comparison principle}\label{cp}

In this section a comparison principle
for~(\ref{v:u})--(\ref{v:n}) will be proved. We begin with
definitions of a supersolution, a subsolution and a maximal
solution of~(\ref{v:u})--(\ref{v:n}). Let
$Q_T=\Omega\times(0,T),\;S_T=\partial\Omega\times(0,T)$,
$\Gamma_T=S_T\cup\overline\Omega\times\{0\}$, $T>0$.
\begin{definition}\label{v:sup}
We say that a nonnegative function $u(x,t)\in C^{2,1}(Q_T)\cap
C^{1,0}(Q_T\cup\Gamma_T)$ is a supersolution
of~(\ref{v:u})--(\ref{v:n}) in $Q_{T}$ if
        \begin{equation}\label{v:sup^u}
u_t \geq \Delta u + a u^p \int_{\Omega} u^q (y,t) \, dy - b u^m,\;(x,t)\in Q_T,
        \end{equation}
        \begin{equation}\label{v:sup^g}
\frac{\partial u(x,t)}{\partial\nu}\geq\int_{\Omega}{k(x, y,
t)u^l(y, t) }\,dy, \; x \in \partial \Omega,\; 0 < t < T,
        \end{equation}
        \begin{equation}\label{v:sup^n}
            u(x,0)\geq u_{0}(x),\; x\in\Omega,
        \end{equation}
and $u(x,t)\in C^{2,1}(Q_T)\cap C^{1,0}(Q_T\cup\Gamma_T)$ is a
subsolution of~(\ref{v:u})--(\ref{v:n}) in $Q_{T}$ if $u\geq0$ and
it satisfies~(\ref{v:sup^u})--(\ref{v:sup^n}) in the reverse
order. We say that $u(x,t)$ is a solution of
problem~(\ref{v:u})--(\ref{v:n}) in $Q_T$ if $u(x,t)$ is both a
subsolution and a supersolution of~(\ref{v:u})--(\ref{v:n}) in
$Q_{T}$.
\end{definition}
\begin{definition}\label{v:max1}
We say that a solution $u(x,t)$ of~(\ref{v:u})--(\ref{v:n}) in
$Q_{T}$ is a maximal solution if for any other solution $v(x,t)$
of~(\ref{v:u})--(\ref{v:n}) in $Q_{T}$ the inequality $v(x,t)\leq
u(x,t)$ is satisfied for $(x,t)\in Q_T\cup\Gamma_T$.
\end{definition}

\begin{theorem}\label{Th3} Let $\overline{u}$ and $\underline{u}$ be a
 supersolution and a  subsolution of problem
(\ref{v:u})--(\ref{v:n}) in $Q_T,$ respectively. Suppose that
$\underline{u}(x,t)> 0$ or $\overline{u}(x,t) > 0$ in ${Q}_T\cup
\Gamma_T$ if $\min (p, q, l) < 1.$ Then $ \overline{u}(x,t) \geq
\underline{u}(x,t) $ in ${Q}_T\cup \Gamma_T.$
\end{theorem}
\begin{proof}
Suppose that $\min \{p,q,l\} \geq 1.$  Let $T_0 \in (0,T)$ and $\{\varepsilon_n\}$ be decreasing to $0$ a sequence such that
$0<\varepsilon_n<1.$ For $\varepsilon=\varepsilon_n$ let
$u_{0\varepsilon}(x)$ be the functions with the following
properties: $u_{0\varepsilon}(x) \in C^1(\overline\Omega),\,$
$u_{0\varepsilon}(x) \ge \varepsilon,\,$ $u_{0\varepsilon_i}(x)
\ge u_{0\varepsilon_j}(x)$ for $\varepsilon_i>\varepsilon_j, \,$
$u_{0\varepsilon}(x) \to u_{0}(x)$ as $\varepsilon \to 0$ and
$$
\frac{\partial u_{0\varepsilon} (x)}{\partial\nu} = \int_{\Omega}
k(x,y,0) u^l_{0\varepsilon}(y) \, dy
$$
for $x \in \partial \Omega.$ Let $u_\varepsilon$ be the solution of the following auxiliary
problem:
\begin{eqnarray} \label{E:2.1}
\left\{ \begin{array}{ll}
 u_t= \Delta u + a u^p \int_{\Omega} u^q (y,t) \, dy - b u^m + b \varepsilon^m
 \,\,\,&\textrm{for} \,\,\, x \in \Omega, \,\,\,\,\, t > 0, \\
\frac{\partial u(x,t)}{\partial\nu} = \int_{\Omega} k(x,y,t) u^l (y,t) dy \,\,\,
& \textrm{for} \,\,\, x \in \partial \Omega, \,\, t > 0,  \\
u(x,0)= u_{0\varepsilon}(x)  \,\,\,& \textrm{for} \,\,\, x \in
\Omega,
\end{array} \right.
\end{eqnarray}
where $\varepsilon=\varepsilon_n.$ It was proved in \cite{WY} (see also \cite{G2} for similar problem) that
$u_m(x,t)=\lim_{\varepsilon \to 0} u_\varepsilon(x,t)$ is a maximal solution of
 (\ref{v:u})--(\ref{v:n}). To establish theorem we will show that
\begin{equation}\label{E:1.1}
\underline{u}(x,t) \leq u_m(x,t) \leq \overline{u}(x,t) \,\,\,
 \textrm{in} \,\,\, \overline{Q}_{T_0}  \,\,\,
 \textrm{for any} \,\,\, T_0 \in (0,T).
\end{equation}
We prove the second inequality in (\ref{E:1.1}) only, since the
proof of the first one is similar. Let $\varphi (x,t) \in
C^{2,1}(\overline{Q}_{T_0})$ be a nonnegative function such that
$$
\frac{\partial \varphi (x,t)}{\partial \nu} = 0, \; (x,t) \in
S_{T_0}.
$$
If we multiply the first equation in (\ref{E:2.1}) by $\varphi
(x,t)$ and then integrate over $Q_{t}$ for $ 0 < t < T_0,$ we get
\begin{eqnarray}\label{E:1.3}
\hspace{-0.7cm}
\int_\Omega u_\varepsilon(x,t) \varphi(x,t)\, dx
&\leq& \int_\Omega u_\varepsilon(x,0)\varphi (x,0)\, dx +
\varepsilon^m b \int_0^t\int_{\Omega}  \varphi (x,\tau) \, dx d\tau
\nonumber\\
&+& \int_0^t\int_{\Omega} \left[ u_\varepsilon (x,\tau) \varphi_{\tau} (x,\tau) +
u_\varepsilon (x,\tau) \Delta \varphi (x,\tau) \right. \nonumber\\
&+& \left. a u_\varepsilon^p (x,\tau) \int_{\Omega} u_\varepsilon^q (y,\tau) \, dy  - b u_\varepsilon^m \varphi (x,\tau) \right] \,
dx d\tau \nonumber \\
&+&\int_0^t\int_{\partial\Omega} \varphi (x,\tau) \int_\Omega
k(x,y,\tau) u_\varepsilon^l(y,\tau)\, dy  \, dS_x d\tau.
\end{eqnarray}
On the other hand, $\overline{u}$ satisfies (\ref{E:1.3}) with
reversed inequality and with $\varepsilon =0.$ Set
$w(x,t)=u_\varepsilon(x,t) - \overline{u}(x,t).$ Then $w(x,t)$
satisfies
\begin{eqnarray}\label{E:1.5}
\hspace{-0.9cm}
\int_\Omega w(x,t)\varphi (x,t)\, dx &\leq&
\int_\Omega w(x,0)\varphi (x,0)\, dx + \varepsilon^m b
\int_0^t\int_{\Omega}  \varphi \, dx d\tau\nonumber\\
&+& \int_0^t\int_{\Omega} (\varphi_{\tau} + \Delta \varphi -
 b m \theta_1^{m-1} (x,\tau) \varphi) w \, dx d\tau \nonumber \\
&+& \int_0^t\int_{\Omega}  a p \theta_2^{p-1} (x,\tau) w (x,\tau) \int_{\Omega} u_\varepsilon^q(y,\tau)\, dy
 dx d\tau \nonumber \\
&+& \int_0^t\int_{\Omega}  a \overline{u}^p (x,\tau) \int_\Omega q \theta_3^{q-1} (y,\tau) w  (y,\tau)\, dy  dx d\tau \nonumber \\
 &+& \int_0^t\int_{\partial\Omega} \varphi (x,\tau) \int_\Omega
k(x,y,\tau) l \theta_4^{l-1}w(y,\tau)\, dy dS_x d\tau,
\end{eqnarray}
where  $\theta_i, \, i = \overline{1,4}$ are some continuous functions
between $u_\varepsilon$ and $\overline{u}.$ Note here that by
hypotheses for  $k(x,y,t)$, $u_\varepsilon(x,t)$ and
$\overline{u}(x,t)$, we have
\begin{eqnarray}\label{E:1.6}
&&  0 \leq \overline{u}(x,t) \leq M,
\,\, \varepsilon \leq u_\varepsilon(x,t) \leq M \,\,
\textrm{in} \,\,  \overline{Q}_{T_0} \nonumber \\
&&\textrm{and}\,\,\,0 \leq k(x,y,t) \leq M \,\,\, \textrm{in}\,\,
\partial \Omega \times \overline{Q}_{T_0},
\end{eqnarray}
where $M$ is some positive constant. Then it is easy to see from
(\ref{E:1.6}) that $\theta_1^{m-1},\,$ $\theta_2^{p-1},\,$ $\theta_3^{q-1}$ and $\theta_4^{l-1}$ are
positive and bounded functions in $\overline{Q}_{T_0}$ and,
moreover, $  \theta_2^{p-1} \leq  M^{p-1},\,$ $\theta_3^{q-1} \leq  M^{q-1},\,$ $ \theta_4^{l-1} \leq  M^{l-1}.$
Define a sequence $\{a_k
\}$ in the following way: $\, a_k \in C^\infty
(\overline{Q}_{T_0}),\,$ $ a_k \geq 0 \,$ and $\, a_k \to b m
\theta_1^{m-1}\,$ as $\, k \to \infty \,$ in $L^1({Q}_{T_0}).$
Now consider a backward problem given by
\begin{eqnarray}\label{E:1.7}
\left\{ \begin{array}{ll}
\varphi_{\tau} + \Delta \varphi - a_k
\varphi = 0 \,\,\, & \textrm{for} \,\,\,x \in \Omega, \,\,\,
0<\tau<t, \\
\frac{\partial \varphi (x,\tau)}{\partial \nu} = 0 \,\,\,&
\textrm{for} \,\,\, x \in
\partial\Omega, \,\,\, 0 < \tau < t, \\
\varphi(x,t)= \psi (x) \,\,\,& \textrm{for} \,\,\,x \in \Omega,
\end{array} \right.
\end{eqnarray}
where $\psi (x)\in C_0^\infty (\Omega)$ and $ 0 \leq \psi (x) \leq
1.$ Denote a solution of (\ref{E:1.7}) as $\varphi_k (x,\tau).$
Then by the standard theory for linear parabolic equations (see
\cite{LSU}, for example), we find that  $\varphi_k \in
C^{2,1}(\overline{Q}_{t}),$  $0 \leq \varphi_k (x,\tau) \leq 1$ in
$\overline{Q}_{t}.$  Putting $\varphi = \varphi_k$ in
(\ref{E:1.5}) and passing then to the limit as $k \to \infty,$ we
infer
\begin{eqnarray}\label{E:1.8}
 \int_\Omega w(x,t)\psi (x)\, dx &\leq& \int_\Omega w(x,0)_+ \, dx +  \varepsilon^m b  T_0 \vert \Omega \vert  \nonumber\\
&+& [a (p+q) \vert\Omega\vert M^{p+q-1} +  l \vert \partial\Omega \vert M^l ] \int_0^t
\int_\Omega w(y,\tau)_+ \, dy d\tau,
\end{eqnarray}
where $w_+ = \max \{w,0 \},$ $\vert \partial \Omega\vert$ and $\vert \Omega\vert $ are
the Lebesgue measures of $\partial \Omega$ in $\mathbb R^{N-1}$ and
$\Omega$ in $\mathbb R^N,$ respectively. Since (\ref{E:1.8}) holds
for every $\psi (x),$ we can choose a sequence $\{ \psi_k \}$
converging in $L^1(\Omega)$ to $\psi (x) =1$ if $w(x,t) > 0$ and $\psi
(x) = 0$ otherwise. Hence, from (\ref{E:1.8}) we get
\begin{eqnarray*}
\int_\Omega w(x,t)_+ \, dx  &\leq& \int_\Omega w(x,0)_+ \, dx +  \varepsilon^m b  T_0 \vert \Omega\vert  \nonumber\\
&+& [a (p+q) \vert\Omega\vert M^{p+q-1} +  l \vert\partial\Omega\vert M^l ] \int_0^t
\int_\Omega w(y,\tau)_+ \, dy d\tau.
\end{eqnarray*}
Applying now Gronwall's inequality and passing to the limit
$\varepsilon \to 0,$ the conclusion of the theorem follows for  $\min (p,q,l) \geq 1.$
For the case $\min (p,q,l) < 1$ we can consider
$w(x,t)=\underline{u}(x,t) - \overline{u}(x,t)$ and prove the
theorem in a similar way using the positiveness of a subsolution
or a supersolution.
\end{proof}

\begin{remark}\label{Rem1}
Theorem~\ref{Th3} holds for $q=0.$
\end{remark}
\begin{remark}
We improve a comparison principle in \cite{WY}. The authors of \cite{WY} supposed that
$\underline{u}(x,t)> 0$ or $\overline{u}(x,t) > 0$ in ${Q}_T\cup
\Gamma_T$ if $\min (m, p, q, l) < 1.$
\end{remark}
Next lemma shows the positiveness for $t>0$  of solutions of~(\ref{v:u})--(\ref{v:n})  with $m \geq 1$ and $u_0 \not\equiv 0.$
\begin{lemma}\label{Lem1}
Let $u_0 \not\equiv 0 $ in $\Omega,$  $m \geq 1.$  Suppose $u$ is a solution of
(\ref{v:u})--(\ref{v:n}) in $Q_T.$ Then $u>0$ in ${Q}_T \cup S_T.$
\end{lemma}
\begin{proof}
We take any $\tau \in (0,T).$
As  $u(x,t)$ is continuous in $\overline{Q}_\tau$ function, then we have
\begin{equation}\label{E:1.10}
 u(x,t) \leq M, \,\, (x,t) \in \overline{Q}_\tau
\end{equation}
with some positive constant $M.$ Now we put $v= u\exp (\lambda t),$
where $\lambda \geq bM^{m-1}.$ It is easy to verify that $v_t - \Delta
v \geq 0.$ Since $v (x,0) = u_0(x)\not\equiv0$ in $\Omega$ and $v
(x,t) \geq 0$ in $Q_\tau,$ by the strong maximum principle $v(x,t)>0$
in $Q_\tau.$ Let $v(x_0,t_0)=0$ in some point $(x_0,t_0)\in S_T.$
Then according to Theorem~3.6 of~\cite{H} it yields $\partial
v(x_0,t_0)/\partial\nu < 0,$ which contradicts~(\ref{v:g}).
\end{proof}

%     Global existence

\section{Global existence}\label{gl}

\begin{theorem}\label{global}
Let $\max (p+q,l) \leq 1$ or $1 < \max (p+q,l) < m.$  Then every solution
of (\ref{v:u})--(\ref{v:n}) is global.
\end{theorem}
\begin{proof}
In order to prove global existence of solutions we construct a
suitable explicit supersolution of~(\ref{v:u})--(\ref{v:n}) in
$Q_T$ for any positive $T.$ Suppose at first that $\max (p+q,l) \leq 1.$
Since $k(x,y,t)$ is a continuous function, there exists a constant
$K>0$ such that
\begin{equation}\label{K}
k(x,y,t)\leq K
\end{equation}
in $\partial\Omega\times Q_T.$ Let $\lambda_1$ be the first
eigenvalue of the following problem
\begin{equation}\label{EF}
    \begin{cases}
        \Delta\varphi+\lambda\varphi=0,\;x\in\Omega,\\
        \varphi(x)=0,\;x\in\partial\Omega,
    \end{cases}
\end{equation}
and $\varphi(x)$ be the corresponding eigenfunction with
$\sup\limits_{\Omega}\varphi(x)=1$. It is well known,
$\varphi(x)>0$ in $\Omega$ and $\max\limits_{\partial\Omega}
\partial\varphi(x)/\partial\nu < 0.$

We now construct a supersolution of~(\ref{v:u})--(\ref{v:n}) in $Q_T$ as follows
\begin{equation*}
\overline u (x,t) = \frac{C\exp (\mu t)}{c \varphi (x) + 1},
\end{equation*}
where constants $C,\mu$ and $c$ are chosen to satisfy the
inequalities:
\begin{equation*}
c\geq \max \left\{ K \int_\Omega \frac{dy}{(\varphi (y) + 1)^l}
\max_{\partial \Omega} \left(-\frac{\partial \varphi}{\partial
\nu} \right)^{-1}, 1 \right\},
\end{equation*}
\begin{equation*}
C \geq \max \{ \sup_\Omega (c \varphi (x) + 1) u_0 (x), 1 \},
\quad \mu \geq \lambda_1 + 2 c^2 \sup_\Omega \frac{\vert\nabla \varphi
\vert^2 } {(c \varphi (x) + 1)^2} + a \vert\Omega\vert.
\end{equation*}
It is easy to obtain
\begin{align}
\overline u_t - \Delta\overline u &- a \overline u^p \int_{\Omega} \overline u^q (y,t) \, dy + b \overline u^m \nonumber \\
&\geq \left( \mu - \frac{c\lambda_1\varphi}{c \varphi (x) + 1} - 2
c^2 \sup_\Omega \frac{\vert\nabla \varphi \vert^2 } {(c \varphi (x) +
1)^2} - a \vert\Omega\vert \right) \overline u  \geq 0 \label{E1}
\end{align}
for $(x,t) \in Q_T,$
\begin{align}
\frac{\partial \overline u}{\partial\nu} =& c C \exp (\mu t)
\left(-\frac{\partial \varphi}{\partial \nu} \right) \geq  K C^l
\exp (l\mu t) \int_\Omega \frac{dy}{(\varphi (y) + 1)^l} \nonumber \\
\geq& \int_{\Omega} k(x,y,t)\overline u^l(y,t) \,dy \label{E2}
\end{align}
for $(x,t) \in S_T$ and
\begin{equation}\label{E3}
\overline u(x,0)\geq u_0(x)
\end{equation}
for $x \in \Omega.$ By virtue of (\ref{E1})--(\ref{E3}) and Theorem~\ref{Th3},
the solution of (\ref{v:u})--(\ref{v:n}) exists globally.

Suppose now that $1 < \max (p+q,l) < m.$ Let $v_0 (x)$ be the function with the following properties:
\begin{equation*}
v_0(x)\in C^1(\overline{\Omega}),\;v_0(x) > u_0(x)\textrm{ in
}\overline \Omega,\;\frac{\partial v_0(x)}{\partial\nu}=\int_{\Omega}{K v_0^l(y)}\,dy \textrm{ on } \partial\Omega,
\end{equation*}
where $K$ was defined in (\ref{K}). Let $v(x,t)$ be the solution of  (\ref{v:u}) in $Q_T$ with
boundary and initial data
\begin{equation}\label{g}
\frac{\partial u(x,t)}{\partial\nu}=
\int_{\Omega}{K u^l(y,t)}\,dy, \; x\in\partial\Omega, \; 0 < t < T,
\end{equation}
\begin{equation}\label{n}
    u(x,0)=v_{0}(x),\; x\in\Omega.
\end{equation}
We prove that $v (x,t)$ exists in $Q_T$ for any $T > 0.$
 By Lemma~\ref{Lem1} and Theorem~\ref{Th3}
\begin{equation*}
u (x,t) \le v(x,t) \textrm{ in } Q_T.
\end{equation*}
 Integrating (\ref{v:u}) for $v (x,t)$ over $Q_t$ and using Green's identity  and H{$\ddot o$}lder's inequality, we have
\begin{align}
& \int_\Omega  v (x,t) \,dx =  \int_\Omega  v_0 (x) \,dx + \int_0^t \int_{\partial \Omega} \int_\Omega K v^l(y,\tau) \,dy \,dS_x \, d\tau  \nonumber \\
&+ a \int_0^t \int_\Omega v^p (y,\tau) \, dy \int_{\Omega} v^q (y,\tau) \, dy \, d\tau -  b \int_0^t \int_\Omega v^m (y,\tau) \, dy \, d\tau  \nonumber \\
&\leq  \int_\Omega  v_0 (x) \,dx + K \vert\partial\Omega\vert (t\vert\Omega\vert)^\frac{m-l}{m}  \left( \int_0^t  \int_\Omega   v^m (y,\tau) \, dy \, d\tau \right)^\frac{l}{m}  \nonumber \\
& + a t^\frac{m-p-q}{m} \vert\Omega\vert^\frac{2m-p-q}{m} \left( \int_0^t  \int_\Omega   v^m (y,\tau) \, dy \, d\tau \right)^\frac{p+q}{m}
- b \int_0^t  \int_\Omega   v^m (y,\tau) \, dy \, d\tau. \label{E4}
\end{align}
Since $v(x,t) > 0,$ we obtain from (\ref{E4}) that
\begin{equation}\label{E5}
\int_0^t  \int_\Omega   v^m (y,\tau) \, dy \, d\tau \leq C(T), \; t \leq T.
\end{equation}

Now we prove that $v(x,t) $ can not blow up in $\Omega.$

Let $G_N(x,y;t-\tau)$ be the Green function of
the heat equation with homogeneous Neumann boundary condition.
Then we have the representation formula:
\begin{align}
        v(x,t)&=\int_\Omega{G_N(x,y;t)v_0(y)} dy + K \int_0^t{\int_{\partial\Omega}G_N(x,\xi;t-\tau) \int_{\Omega}  v^l(y,\tau)} dy dS_\xi d\tau  \nonumber  \\
        &+ a \int_0^t{\int_\Omega G_N(x,y;t-\tau) v^p(y,\tau)\int_\Omega v^q(z,\tau)}\,dz  \,dy d\tau  \nonumber \\
        &- b \int_0^t{\int_\Omega} G_N(x,y;t-\tau) v^m(y,\tau)\,dy d\tau \label{blow:equat}
\end{align}
for $(x,t) \in Q_T.$ We now take a sequence of sets $\{ \Omega_k \}$ such that
$\Omega_k\subset\subset\Omega, \,$  $\Omega_k\subset\Omega_{k+1}, \,$
 $\Omega = \cup_{k=1}^{\infty} \Omega_k, \,$ $\partial\Omega_k\in C^2.$
 It is well known (see, for example,~\cite{H},\cite{Kahane}) that
\begin{equation}\label{gl:1G_N}
        G_N(x,y;t-\tau)\geq0,\;x,y\in\Omega,\;0\leq\tau<t<T,
    \end{equation}
    \begin{equation}\label{gl:2G_N}
    \int_{\Omega}{G_N(x,y;t-\tau)}\,dy=1,\;x\in\Omega,\;0\leq\tau<t<T,
    \end{equation}
    \begin{equation}\label{blow:G_N}
    0\leq G_N(x,y;t-\tau)\leq
    c_0,\;x\in\Omega_k,\;y\in\partial\Omega,\;0<\tau<t<T,
    \end{equation}
       \begin{equation}\label{G1}
    c_1 \leq G_N(x,y;t-\tau)\leq
    c_2,\;x\in \overline\Omega,\; y\in\overline\Omega,\; t - \tau \ge \varepsilon_0,
    \end{equation}
          \begin{equation}\label{G2}
    \frac{c_3}{\sqrt{t - \tau}} \leq \int_{\partial\Omega}G_N(x,\xi;t-\tau) dS_\xi
    \leq   \frac{c_4}{\sqrt{t - \tau}},\;x\in \partial\Omega,\; 0 < t - \tau \le \varepsilon_0,
    \end{equation}
where $\varepsilon_0 >0, \,$ $c_0$ is a positive constant depending on
$k.$ Here and below $c_i, \, i \in \mathbb{N}$ are positive constants.
We note that
\begin{equation}\label{2}
     v^p(y,t)\int_\Omega v^q(z,t) \,dz \leq v^{p+q}(y,t) + \left( \int_\Omega v^q(z,t) \,dz \right)^\frac{p+q}{q} \, \textrm{for} \,\,\, (y, t) \in Q_T.
\end{equation}
Indeed, if
\begin{equation*}
\int_\Omega v^q(z,t) \,dz \leq v^{q}(y,t),
\end{equation*}
then
\begin{equation*}
v^p(y,t)\int_\Omega v^q(z,t) \,dz \leq v^{p+q}(y,t).
\end{equation*}
Otherwise, we have
\begin{equation*}
 v^p(y,t)  \leq  \left( \int_\Omega v^q(z,t) \,dz \right)^\frac{p}{q}
\end{equation*}
and
\begin{equation*}
  v^p(y,t)\int_\Omega v^q(z,t) \,dz  \leq  \left( \int_\Omega v^q(z,t) \,dz \right)^\frac{p+q}{q}.
\end{equation*}
Now using (\ref{gl:2G_N}), (\ref{2}) and H$\ddot{o}$lder's inequality, we estimate the third integral in the right hand side of (\ref{blow:equat}):
\begin{align}
        & a \int_0^t{\int_\Omega G_N(x,y;t-\tau)  v^p(y,\tau)\int_\Omega v^q(z,\tau)}\,dz  \,dy d\tau  \nonumber \\
        &\leq a \int_0^t{\int_\Omega} G_N(x,y;t-\tau) \left[ v^{p+q}(y,\tau) + \left( \int_\Omega v^q(z,\tau) \,dz \right)^\frac{p+q}{q} \right] \,dy d\tau  \nonumber \\
        &\leq \int_0^t{\int_\Omega} G_N(x,y;t-\tau) \left[ b v^{m}(y,\tau) +  a \left( \frac{a}{b} \right)^\frac{p+q}{m-(p+q)} \right.  \nonumber \\
        &+  \left. a \vert\Omega\vert^\frac{(m-q)(p+q)}{mq} \left( \int_\Omega v^m(z,\tau) \,dz \right)^\frac{p+q}{m} \right] \,dy d\tau  \nonumber \\
        &\leq b  \int_0^t{\int_\Omega} G_N(x,y;t-\tau) v^m (y,\tau) \,dy d\tau + a \left( \frac{a}{b} \right)^\frac{p+q}{m-(p+q)} t  \nonumber \\
        &+ a\vert\Omega\vert^\frac{(m-q)(p+q)}{mq} \int_0^t \left( \int_\Omega v^m(z,\tau) \,dz \right)^\frac{p+q}{m} \,d\tau. \label{3}
\end{align}
 By~(\ref{E5}) -- (\ref{blow:G_N}), (\ref{3}) and H$\ddot{o}$lder's inequality we have
\begin{align}
        \sup_{\Omega_k} v(x,t) \leq& \sup_\Omega v_0(x) + c_0 \vert\partial \Omega\vert K \int_0^t \int_\Omega v^l (y,\tau) \, dy  d\tau \nonumber \\
       +& c_5 T + c_6 T^\frac{m-p-q}{m} \left[  \int_0^t \int_\Omega v^m (y,\tau) \, dy  d\tau \right]^\frac{p+q}{m}
       \leq c_7 (T). \label{E6}
\end{align}

Suppose that $v(x,t)$ blows up at  $t = T $ and
\begin{equation}\label{N0}
\int_\Omega v^q (y,T) \, dy = +\infty.
\end{equation}
Then by H$\ddot{o}$lder's inequality we derive that
\begin{equation}\label{N1}
\int_\Omega v^m (y,T) \, dy = +\infty.
\end{equation}
Since $v(x,t)$ blows up at  $t = T $ from  (\ref{blow:equat}), (\ref{G1}), (\ref{G2}), (\ref{3})  we obtain
\begin{equation}\label{N2}
\int_\Omega v^l (y,T) \, dy = +\infty
\end{equation}
and, moreover,
\begin{equation}\label{N21}
v(x,t)  \,  \textrm{blows up at } \,  t = T \,  \textrm{at every point of  } \, \partial \Omega.
\end{equation}
Integrating (\ref{v:u}) for $v(x,t)$ over $\Omega$ and using H{$\ddot o$}lder's inequality, we have
\begin{align*}
\int_\Omega  v_t (x,t) \,dx & =  K \int_{\partial \Omega} \int_\Omega v^l(y,t) \,dy \,dS_x + a \int_\Omega v^p (y,t) \, dy \int_{\Omega} v^q (y,t) \, dy   \\
&-  b \int_\Omega v^m (y,t) \, dy  \leq  K \vert\partial \Omega\vert \vert\Omega\vert^\frac{m-l}{m} \left(  \int_\Omega   v^m (y,t) \, dy  \right)^\frac{l}{m}  \\
& + a \vert\Omega\vert^\frac{2m-p-q}{m} \left(  \int_\Omega   v^m (y,t) \, dy  \right)^\frac{p+q}{m} - b \int_\Omega   v^m (y,t) \, dy.
\end{align*}
Thus by (\ref{N1})
\begin{equation}\label{N3}
\lim_{t \to T} \int_\Omega  v_t (x,t) \,dx = -\infty.
\end{equation}
On the other hand, integrating (\ref{v:u}) for $v(x,t)$ over $\Omega_k$, we infer
\begin{equation*}
\int_{\Omega_k}  v_t (x,t) \,dx =  \int_{\partial \Omega_k} \frac{\partial v(x,t)}{\partial\nu} \,dS_x + a \int_{\Omega_k}  v^p (y,t) \, dy \int_{\Omega}  v^q (y,t) \, dy
- b \int_{\Omega_k}   v^m (y,t) \, dy
\end{equation*}
and using (\ref{g}), (\ref{E6}), (\ref{N0}), (\ref{N2}), (\ref{N21}) we obtain the contradiction  with  (\ref{N3}).

Hence,
\begin{equation}\label{E7}
 \int_\Omega v^q (y,t) \, dy \leq c_8 (T), \; t \leq T.
\end{equation}

Now we consider the following equation
\begin{equation}\label{E8}
Lu \equiv u_t - \Delta u - a c_8 (T) u^p + b u^m = 0,\;(x, t) \in Q_T.
\end{equation}
It is easy to see that $v(x,t)$ is the subsolution of the problem (\ref{E8}), (\ref{g}), (\ref{n}) in $Q_T.$
To construct a supersolution we use the change of variables in a
	neighborhood of $\partial \Omega$ as in \cite{CPE}. Let
	$\overline x\in\partial \Omega$ and $\widehat{n}
	(\overline x)$ be the inner unit normal to $\partial \Omega$ at the
	point $\overline x.$ Since $\partial \Omega$ is smooth it is well
	known that there exists $\delta >0$ such that the mapping $\psi
	:\partial \Omega \times [0,\delta] \to \mathbb{R}^n$ given by
	$\psi (\overline x,s)=\overline x +s\widehat{n} (\overline x)$
	defines new coordinates ($\overline x,s)$ in a neighborhood of
	$\partial \Omega$ in $\overline\Omega.$ A straightforward
	computation shows that, in these coordinates, $\Delta$ applied to
	a function $g(\overline x,s)=g(s),$ which is independent of the
	variable $\overline x,$ evaluated at a point $(\overline x,s)$ is
	given by
	\begin{equation}\label{Gl:new-coord}
	\Delta g(\overline x,s)=\frac{\partial^2g}{\partial s^2}(\overline x,s)-\sum_{j=1}^{n-1}\frac{H_j(\overline x)}{1-s
		H_j (\overline x)}\frac{\partial g}{\partial s}(\overline x,s),
	\end{equation}
	where $H_j (\overline x)$ for $j=1,...,n-1,$ denote the principal
	curvatures of $\partial\Omega$ at $\overline x.$ For $0\leq s\leq \delta$
	and small $\delta$  we have
	\begin{equation}\label{Gl:enq4}
	\left\vert\sum_{j=1}^{n-1} \frac{H_j (\overline x)}{1-s H_j (\overline
		x)}\right\vert\leq\overline c.
	\end{equation}
	
	Let $0<\varepsilon<\omega<\min(\delta, 1), $
	$\max(1/l, 2/(m-1)) <\beta< 2/(l-1),$  $0<\gamma<\beta/2,$ $A\ge\sup_\Omega v_0(x).$
	As in \cite{GK4} for points in $Q_{\delta,T}=\partial \Omega \times [0,
	\delta]\times [0,T]$ of coordinates $(\overline x,s,t)$ define
	\begin{equation}\label{Gl:function}
	\overline v (x,t)=  \overline v ((\overline x,s),t)=\left[(s+\varepsilon)^{-\gamma}-\omega^{-\gamma}\right]_+^\frac{\beta}{\gamma}+A,
	\end{equation}
	where $s_+=\max(s,0).$ For points in $\overline{Q_T}\setminus Q_{\delta,T}$
	we set  $ \overline v(x,t)= A.$ We prove that $ \overline v(x,t)$
	is the supersolution of (\ref{E8}), (\ref{g}), (\ref{n}) in $Q_T.$
	It is not difficult to check that
	\begin{equation}\label{Gl:enq1}
	\left\vert\frac{\partial \overline v}{\partial s}\right\vert\leq \beta\min\left(\left[ D(s)\right]^\frac{\gamma+1}{\gamma}\left[( s+\varepsilon)^{-\gamma}-\omega^{-\gamma}\right]_+^\frac{\beta+1}{\gamma},\,(s+\varepsilon)^{-(\beta+1)}\right),
	\end{equation}
	\begin{equation}\label{Gl:enq2}
	\left\vert\frac{\partial^2 \overline v}{\partial s^2}\right\vert\leq \beta (\beta+1)\min\left(\left[D(s)\right]^\frac{2(\gamma+1)}{\gamma}\left[( s+\varepsilon)^{-\gamma}-
	\omega^{-\gamma}\right]_+^\frac{\beta+2}{\gamma},\,( s+\varepsilon)^{-(\beta+2)}\right),
	\end{equation}
	where
	\begin{equation*}
	D(s)= \frac{( s+\varepsilon)^{-\gamma}}{ (s+\varepsilon)^{-\gamma}-\omega^{-\gamma}}.
	\end{equation*}
	Then $D^\prime(s)>0$ and for any $\overline\varepsilon>0$
	\begin{equation}\label{Gl:enq3}
	1\leq D(s)\leq 1+\overline\varepsilon, \; 0<s\leq{\overline s},
	\end{equation}
	where ${\overline s} = [\overline\varepsilon/(1+\overline\varepsilon)]^{1/\gamma}\omega-\varepsilon,$
	$\varepsilon<[\overline\varepsilon/(1+\overline\varepsilon)]^{1/\gamma}\omega.$
By (\ref{E8})--(\ref{Gl:enq3})  we can choose
	$\overline\varepsilon$ small and  $A$ large so that in $Q_{{\overline s},T}$
\begin{align*}
        &L\overline v \geq   b \left( \left[( s+\varepsilon)^{-\gamma}-\omega^{-\gamma}\right]_+^\frac{\beta }{\gamma} + A \right)^m - \beta(\beta+1)\left[ D(s)\right]^\frac{2(\gamma+1)}{\gamma}\left[( s+\varepsilon)^{-\gamma}-\omega^{-\gamma}\right]_+^\frac{\beta+2}{\gamma} \\
       &-  \beta \overline c \left[ D(s)\right]^\frac{\gamma+1}{\gamma} \left[( s+\varepsilon)^{-\gamma}-\omega^{-\gamma}\right]_+^\frac{\beta+1}{\gamma}
       -a c_8 (T) \left(\left[( s+\varepsilon)^{-\gamma}-\omega^{-\gamma}\right]_+^\frac{\beta}{\gamma}+A\right)^p \geq 0.
      \end{align*}
	
	Let $s\in[{\overline s},\delta].$ From (\ref{Gl:new-coord})--(\ref{Gl:enq2}) we have
	\begin{equation*}
	\vert\Delta \overline v\vert \leq \beta(\beta+1)\omega^{-(\beta+2)}\left(\frac{1+\overline\varepsilon}
	{\overline\varepsilon}\right)^\frac{\beta+2}{\gamma} + \beta\overline c\omega^{-(\beta+ 1)}\left(\frac{1+\overline\varepsilon}{\overline\varepsilon}\right)^\frac{\beta+1}{\gamma}
	\end{equation*}
	and  $L\overline v\geq0$ for large values of $A.$ Obviously, in $\overline{Q_T}\setminus Q_{\delta,T}$
	\begin{equation*}
	L\overline v = - a c_8 (T) A^p + b A^m  \geq  0
	\end{equation*}
	for large values of $A.$
	
	Now we prove the following inequality
\begin{equation}\label{E:4.6}
\frac{\partial \overline v}{\partial\nu} (\overline x,0,t) \geq
\int_{\Omega} K \overline v^l(\overline x,s,t) \, dy, \quad
(x,t) \in S_T
\end{equation}
for a suitable choice of $\varepsilon.$ To estimate the integral
$I$ in the right hand side of (\ref{E:4.6}) we use the change of
variables in a neighborhood of $\partial\Omega$ as above. Let
	\begin{equation*}
 \overline J= \sup_{0< s< \delta} \int_{\partial\Omega}
\vert J(\overline y,s)\vert \, d\overline y,
	\end{equation*}
where $J(\overline y,s)$ is Jacobian of the change of variables.
Then we have
\begin{align*}
I \leq & \theta  K \int_{\Omega} \left[ ( s +
\varepsilon)^{-\gamma} - \omega^{-\gamma}\right]_+^\frac{\beta
l}{\gamma} \, dy + \theta  K A^l \vert\Omega\vert\\
\leq & \theta  K  \overline J \int_{0}^{\omega -
\varepsilon} \left[ ( s + \varepsilon)^{-\gamma} -
\omega^{-\gamma}\right]^\frac{\beta
l}{\gamma} \, ds + \theta  K A^l \vert\Omega\vert\\
\leq & \frac{\theta  K  \overline J}{\beta l-1} \left[
\varepsilon^{-(\beta l-1)} - \omega^{-(\beta l-1)}\right] +
\theta  K A^l \vert\Omega\vert,
 \end{align*}
where $\theta = \max (2^{l-1}, 1). $   On the other hand, since
	\begin{equation*}
\frac{\partial \overline v}{\partial\nu} (\overline x,0,t) = -
\frac{\partial \overline v}{\partial s} (\overline x,0,t) =
\beta \varepsilon^{-\gamma -1} \left[ \varepsilon^{-\gamma} -
\omega^{-\gamma}\right]_+^\frac{\beta-\gamma}{\gamma},
	\end{equation*}
the inequality (\ref{E:4.6}) holds if $\varepsilon$ is small
enough. At last,
	\begin{equation*}
 v(x,0) \leq \overline v (x,0) \,\,\,  \textrm{in} \,\,\,  \Omega.
	\end{equation*}
Hence, by comparison principle for (\ref{E8}), (\ref{g}), (\ref{n}) (see Remark~\ref{Rem1}) we get
	\begin{equation*}
  v(x,t) \leq \overline v (x,t) \,\,\,  \textrm{in} \,\,\,
 \overline{Q}_T.
	\end{equation*}
\end{proof}

\begin{remark}
The authors of \cite{WY} have proved global existence of solutions of (\ref{v:u})--(\ref{v:n}) for  $p+q < m, \, l \leq 1.$
\end{remark}

\section{Blow-up in finite time}\label{blow}

To formulate finite time blow-up result we introduce
\begin{equation*}
\underline k(t) = \inf_{\Omega} \int_{\partial \Omega} k(x,y,t) \, dS_x
\end{equation*}
 and suppose that
\begin{equation}\label{E9}
\underline k(0) > 0.
\end{equation}

\begin{theorem}\label{blow-up}
Let either  $r+p>\max(m,1)$ or $l>\max(m,1)$ and (\ref{E9})  hold. Then solutions of~(\ref{v:u})--(\ref{v:n}) blow up in finite time if initial data are large enough.
\end{theorem}
\begin{proof}
Suppose at first that $r+p>\max(m,1).$ To prove the existence finite time blow-up solutions we construct a suitable subsolution.
Let $f(t)$ be the solution of the following problem
\begin{equation*}
    \begin{cases}
    f'(t) = a \vert \Omega \vert f^{p+q} - b f^m, \; t >0, \\
        f(0) = f_0.
    \end{cases}
\end{equation*}
If
$$
f_0 > \left( \frac{b}{a \vert \Omega \vert} \right)^ \frac{1}{p+q-m},
$$
then $ f(t) \geq f_0$ and there exists $t_0 < +\infty$ such that
$$
\lim_{t \to t_0}  f(t) = +\infty.
$$
It is easy to check that $f(t)$  is a subsolution of~(\ref{v:u})--(\ref{v:n}) if $u_0 (x) \geq f_0.$ By Theorem~\ref{Th3}
$u(x,t)$ blows up in finite time.

Now suppose $l>\max(m,1)$ and (\ref{E9})  holds. Then there exists $T_0 >0$ such that $\underline k(t) > 0$ for $ t \in [0, T_0].$ Denote
$$
k_0 = \min_{[0, T_0]} \underline k(t), \quad V(t) =  \int_{\Omega} u(x,t) \, dx.
$$
Integrating (\ref{v:u}) over $\Omega$ and using Green's identity  and H{$\ddot o$}lder's inequality, we obtain for $t \leq T_0$
\begin{equation}\label{E10}
\begin{split}
& \int_\Omega  u_t (x,t) \, dx =  \int_{\partial \Omega} \int_\Omega k(x,y,t) u^l(y,t) \,dy \,dS_x + a \int_\Omega u^p (y,t) \, dy \int_{\Omega} u^q (y,t) \, dy \\
&-  b \int_\Omega u^m (y,t) \, dy \geq   k_0 \int_\Omega   u^l (y,t) \, dy -  b  \int_\Omega u^m (y,t) \, dy \\
&\geq \left( \int_\Omega   u^l (y,t) \, dy \right)^\frac{m}{l}  \left[ k_0 \left( \int_\Omega u^l (y,t) \, dy \right)^\frac{l-m}{l} - b \vert\Omega\vert^\frac{l-m}{l} \right] \\
&\geq \left( \int_\Omega   u (y,t) \, dy \right)^{m}  \vert\Omega\vert^{-\frac{m(l-1)}{l}}  \left[ k_0 \left( \int_\Omega   u (y,t) \, dy \right)^{l-m} \vert\Omega\vert^{-\frac{(l-m)(l-1)}{l}}
- b \vert\Omega\vert^\frac{l-m}{l}\right]
\end{split}
\end{equation}
if
\begin{equation}\label{E11}
J (t) \equiv k_0 V^{l-m} (t) \vert\Omega\vert^{-\frac{(l-m)(l-1)}{l}} - b \vert\Omega\vert^{\frac{l-m}{l}} > 1.
\end{equation}
Then from (\ref{E10}), (\ref{E11})  we have
\begin{equation}\label{E12}
V'(t) \geq  \vert\Omega\vert^{-\frac{m(l-1)}{l}}  V^m (t)
\end{equation}
for $t \leq T_0.$ Let $m>1.$ Obviously, $V(t)$ blows up at $t \leq T_0$ subject to
\begin{equation}\label{E13}
V(0) \geq \left\{ (m-1) \vert\Omega\vert^{-\frac{m(l-1)}{l}} T_0 \right\}^{-\frac{1}{m-1}}.
\end{equation}
To provide (\ref{E11}) we assume that
\begin{equation}\label{E14}
V(0) > k_0^{-\frac{1}{l-m}} \left( b \vert\Omega\vert^{\frac{l-m}{l}} + 1 \right)^{\frac{1}{l-m}} \vert\Omega\vert^{\frac{l-1}{l}}.
\end{equation}
Easily to check (\ref{E14}) is equivalent to the inequality
 \begin{equation*}
k_0 V^{l-m} (0) \vert\Omega\vert^{-\frac{(l-m)(l-1)}{l}} - b \vert\Omega\vert^{\frac{l-m}{l}} > 1.
\end{equation*}
Hence,
 \begin{equation*}
J (0)  > 1.
\end{equation*}
Suppose there exists $t_1 \in (0, T_0]$ such that $J(t_1) = 1$ and $J(t) > 1$ for $ t \in [0, t_1).$
Since for $ t \in [0, t_1]$ (\ref{E12}) holds, we have $V(t) \geq V(0)$ and
\begin{equation*}
J (t) \geq J (0) > 1.
\end{equation*}
Therefore, if we take initial data to satisfy (\ref{E13}), (\ref{E14}) then any solution of~(\ref{v:u})--(\ref{v:n}) with  $m>1$
blows up at $t \leq T_0.$

Let $m \leq 1.$ From (\ref{E10}) and H{$\ddot o$}lder's inequality we deduce that
\begin{equation}\label{E15}
V'(t) \geq   k_0 \int_\Omega   u^l (y,t) \, dy -  b  \int_\Omega u^m (y,t) \, dy \geq k_0 \vert\Omega\vert^{1-l}  V^l (t) - b \vert\Omega\vert^{1-m}  V^m (t).
\end{equation}
Assume
\begin{equation}\label{E16}
V(0) \ge \max \left\{ 1, k_0^{-\frac{1}{l-m}} \left( b \vert\Omega\vert^{\frac{l-m}{l}} + 1 \right)^{\frac{1}{l-m}} \vert\Omega\vert^{\frac{l-1}{l}} \right\}.
\end{equation}
Then by (\ref{E12})
\begin{equation}\label{E17}
  V(t) \ge 1\,\,\,  \textrm{for} \,\,\,  t \ge 0
\end{equation}
and from (\ref{E15}), (\ref{E17}) we obtain
\begin{equation*}
V'(t) \geq k_0 \vert\Omega\vert^{1-l}  V^l (t) - b \vert\Omega\vert^{1-m} (t)  V (t).
\end{equation*}
Solving this inequality we find that $V(t)$ blows up at $t \leq T_0$ if  (\ref{E16}) holds and
\begin{equation*}
V(0) \ge \left\{ \frac{k_0 \vert\Omega\vert^{1-l}}{b \vert\Omega\vert^{1-m}} [ 1 - \exp (- b(l-1)\vert\Omega\vert^{1-m} T_0)]  \right\}^{-\frac{1}{l-1}}.
\end{equation*}
\end{proof}

\subsection*{Acknowledgements}
This work is supported by the RUDN University Strategic Academic Leadership Program and the state program of fundamental research of Belarus
(grant 1.2.03.1).

\end{document}